\documentclass[12pt]{article}
\usepackage{amsmath,amsthm,amssymb,amsfonts}
\usepackage[numbers,sort&compress]{natbib}
\usepackage{color,colordvi}
\usepackage{soul}
\usepackage{fullpage}
\usepackage[letterpaper,top=2cm,bottom=2cm,left=3cm,right=3cm,marginparwidth=2cm]{geometry}

\usepackage{amsmath,amsfonts,amsthm,mathtools,bm}
\usepackage{graphicx}
\usepackage[colorlinks=true, allcolors=blue]{hyperref}
\usepackage[capitalise]{cleveref}
\usepackage{standalone}
\usepackage{indentfirst}
\usepackage[affil-it]{authblk}
\usepackage{diagbox}
\usepackage{changes}
\usepackage{enumitem}
\usepackage{bbm}
\usepackage{orcidlink}

\newtheorem{theorem}{Theorem}[section]
\newtheorem{lemma}[theorem]{Lemma}

\newtheorem{definition}[theorem]{Definition}
\newtheorem*{fact*}{Fact}

\newtheorem*{claim*}{Claim}

\newcommand{\RR}{\mathbb{R}}

\newcommand{\cD}{\mathcal{D}}

\newcommand{\cP}{\mathcal{P}}

\newcommand{\bigO}{\mathcal{O}}

\DeclareMathOperator{\con}{\mathrm{con}}
\DeclareMathOperator{\inter}{\mathrm{inter}}

\DeclarePairedDelimiter{\card}{\lvert}{\rvert}
\DeclarePairedDelimiter{\set}{\lbrace}{\rbrace}
\DeclarePairedDelimiter{\norm}{\lVert}{\rVert}

\DeclarePairedDelimiter{\floor}{\lfloor}{\rfloor}

\title{Higher Steklov eigenvalues of graphs on surfaces}

\author{Xiongfeng Zhan$^{a,b}$\footnote{email: zhanxfmath@163.com }}
\author{Zhe You$^a$\footnote{email: y30231280@mail.ecust.edu.cn (corresponding author)}~\orcidlink{0009-0006-9769-5372} }
\affil{$^a$School of Mathematics, East China University of Science and Technology, Shanghai 200237, China.}
\affil{$^b$School of Mathematics and Statistics, Beijing Jiaotong University, Beijing 100044, China.}

\date{}

\begin{document}
	\maketitle
	
	\begin{abstract} 
		In this paper, we study  the  higher Steklov eigenvalues of graphs on surfaces. 
        We obtain the upper bound of higher Steklov eigenvalues of a finite graph $G$ with boundary $B$ and genus $g$ by using metrical deformation via probability flows. 
         Our result can be regarded as a discrete analogue of Karpukhin's bound in spectral geometry.
         Moreover, this result implies the upper bound of higher Laplacian eigenvalues given by Kelner, Lee, Price and Teng (Geom. Funct. Anal., 2011). 
	\end{abstract}
	
	
	% \tableofcontents
	
	\section{Introduction}
	Let $\Omega$ be a compact orientable Riemannian manifold with smooth boundary $\partial\Omega$. 
    %Consider the Dirichlet-to-Neumann operator $\cD : C^\infty(\Sigma) \to C^\infty(\Sigma)$ defined by $\cD f := \dfrac{\partial \hat{f}}{\partial n}$, where $n$ is the outward normal along the boundary and $\hat{f} \in C^\infty(\Omega)$ is the unique harmonic extension of $f$ into $\Omega$. 
    The Steklov problem is one of the classical eigenvalue problems in  spectral geometry, which is shown as
	$$\begin{cases}\Delta f(x)=0, &  x\in \Omega, \\ 
		\frac{\partial f}{\partial n}(x) =\sigma f(x), & x\in\partial\Omega,
        \end{cases}$$
	where $\Delta$ denotes the Laplace–Beltrami operator and $\frac{\partial}{\partial n}$ is the unit outward normal derivative along the boundary. This problem can also be regarded as the eigenvalue problem of the Dirichlet-to-Neumann operator $\cD: H^\frac{1}{2}(\partial\Omega)\rightarrow H^{-\frac{1}{2}}(\partial\Omega)$, which is defined as
    $$\cD f(x)\coloneqq\frac{\partial \hat{f}}{\partial n}(x)=\sigma f(x),\quad x\in\partial\Omega.$$
    Here, $\hat{f}$ is the harmonic extension of $f$ into $\Omega$.
Since the Dirichlet-to-Neumann operator is a pseudo-differential operator, its spectrum is discrete and can be ordered as
	\begin{align*}
		0 = \sigma_1 \leq \sigma_2 \leq \cdots.
	\end{align*}
	These eigenvalues are called Steklov eigenvalues. 
 Some results of domains in Euclidean spaces and Riemannian
manifolds were given in \cite{weinstock_inequalities_1954,escobar_geometry_1997,escobar_isoperimetric_1999,brock_isoperimetric_2001,wang2009sharp,karpukhin2014multiplicity,girouard2016steklov,yang2017estimates,fall2017profile,fraser2019shape,chen2024upper,girouard2024large}. Different from the Dirichlet and Neumann problems, the eigenvalue $\sigma$ for the Steklov problem is in the boundary condition.

In this paper, we study the Steklov problem on graphs, which was introduced by \cite{hua_first_2017,hassannezhad_higher_2020}. 
The idea of obtaining the upper bound for the Escobar Cheeger-type constant introduced in~\cite{hua_first_2017} motivates us to find the upper bounds for the Steklov eigenvalues of graphs. 
	 He and Hua~\cite{he_upper_2022} considered the problem on trees and showed that $\sigma_2\leq \frac{4(D-1)}{|B|}$ and $\sigma_2\leq \frac{2}{L}$, where $B$, $D$, and $L$ respectively denote the boundary,  maximum degree, and diameter of the tree. They further provided a result on higher Steklov eigenvalues of trees. 
     For a finite subgraph of the integer lattice, Han and Hua~\cite{han_steklov_2023} proved a discrete version of Brock's result~\cite{brock_isoperimetric_2001}. Recently, Lin and Zhao~\cite{lin_first_2024} obtained that $\sigma_2 \leq \frac{8D}{\card{B}}$ for planar graphs by using circle packing and conformal mapping. They also extended He and Hua's results to block graphs. Lin, Liu, You, and Zhao ~\cite{lin2024upperboundsstekloveigenvalues} further considered $\sigma_2$ of graphs on surfaces.
	 Some other results on upper bounds, lower bounds, Cheeger-type inequalities, rigidity problems, and so on can be found in~\cite{perrin_lower_2019,perrin_isoperimetric_2021,shi_lichnerowicz-type_2022,tschanz2022upper,yu2022minimalstekloveigenvaluescombinatorial,hua2024cheegertypeinequalitiesassociated,lin2024maximizestekloveigenvaluetrees,Yu2024Monotonicity,shi2025extensionrigidityperrinslower}.
	
	The (orientable) genus $g$ of a graph $G$ is the minimum genus of an (orientable) surface in which the graph can be embedded. 
    In this paper, we generalize Lin, Liu, You and Zhao's result as follows.
	\begin{theorem}\label{thm:sigma_k}
		Let $G=(V,E)$ be a finite  connected graph with boundary $B$, and its maximum degree is $D$. Then the $k$-th  Steklov eigenvalue
		\begin{equation}\label{eq:sigma_k}
			\sigma_k(G,B)= 
			\begin{cases}
				\mathcal{O}\left(D\frac{k}{|B|}\right), &\mbox{if $G$ is planar},\\
				\mathcal{O}\left(D g(\mathrm{log}~g)^2\frac{k}{|B|}\right), &\mbox{if $G$ is of genus $g>0$},\\
				\mathcal{O}\left(D h^6\mathrm{log}~h\frac{k}{|B|}\right), &\mbox{if $G$ is $K_h$-minor-free}.\\
			\end{cases}
		\end{equation}
		
	\end{theorem}
    Our result also implies Kelner, Lee, Price and Teng's result~\cite{GAFA} if we choose $B=V$.
	We remark that the $(\mathrm{log}~g)^2$ factor in~\eqref{eq:sigma_k} comes from a certain geometric decomposability property of genus $g$ graphs (see~\cref{thm:genus-g-alpha-bound}) and is most likely non-essential. As a discrete analogue of Karpukhin's bound~\cite[Theorem 1.4]{karpukhin2017BoundsLaplaceSteklov} for a compact orientable Riemann surface $\Omega$ of genus $g$ with boundary $\partial \Omega$ which is shown as 
	\begin{align*}
		\sigma_k \leq \frac{2\pi(g + b + k - 2)}{\card{\partial \Omega}},
	\end{align*}
	where $b$ is the number of boundary components, it is a natural problem whether  $(\mathrm{log}~g)^2$ could be removed and $gk$ could be replaced by $g+k$.

\section{Preliminary}\label{section2}
\subsection{Steklov eigenvalues on graphs}
Let $G = (V, E)$ be an undirected simple graph, where $V, E$ are respectively the set of vertices and edges of $G$. 
A graph with boundary is a pair $(G, B)$, where $B \subset V$ is a non-empty set. 
The set $B$ is called the boundary and $\Omega \coloneqq V \setminus B$ is called  the interior. 
We denote by $E(\Omega_1, \Omega_2) = \set{\set{x,y} \in E \mid x \in \Omega_1, y \in \Omega_2}$ for $\Omega_1, \Omega_2 \subseteq V$. 
Notice that the notion of graphs with boundary here is more general than the definition in~\cite{perrin_lower_2019} since the condition $E(B, B) = \emptyset$ is not required and the boundary vertices can be not adjacent to any interior vertices.
Steklov eigenvalues of such more general graphs with
boundary were considered in~\cite{Colbois2014spectralgap,Colbois2018discretizations,Yu2024Monotonicity}.
In this paper, we assume that $|B| \geq 2$ and $G$ is connected.
	 
     Denote the finite dimensional real vector space of functions on any $\Omega$ by $\mathbb{R}^\Omega\coloneqq\{f:\Omega\rightarrow\mathbb{R}\}$ with  scalar product
$$\langle f,g\rangle_\Omega=\sum\limits_{x\in\Omega}f(x)g(x),\quad \forall f,g\in \mathbb{R}^\Omega.$$
	For $\Omega=V$, we omit the subscript of the scalar product.
    
    We consider the Steklov problem on the pair $(G, B)$. 
	For a real function $f \in \RR^V$ on $V$, the discrete Laplacian operator $\Delta$ can be defined as
	\begin{align*}
		(\Delta f)(x) \coloneqq \sum_{\{x,y\}\in E} (f(x) - f(y)),\quad x\in V,
	\end{align*}
    which is a self-adjoint and positive semi-definite operator.  
    The Steklov eigenvalue problem on graphs can be viewed as the following equations
	\begin{equation*}
		\begin{cases}
			\Delta f(x)=0, & x\in \Omega, \\ 
			 \frac{\partial f}{\partial n} (x)=\sigma f(x), & x\in B ,
		\end{cases}
	\end{equation*}
	where $\frac{\partial f}{\partial n}(x) \coloneqq \sum\limits_{\{x,y\}\in E} (f(x) - f(y))=-\Delta f(x)$ for $x \in B$ and $y\in V$. 
    The reason to define $\frac{\partial f}{\partial n}$ is based on discrete Green’s formula.
    When $B=V$, $\sigma$ is the Laplacian eigenvalue of $G$.
    Here $\sigma \in \mathbb{R}$ is called the Steklov eigenvalue of the graph with boundary $B$. 
	For any function $0 \neq f\in \RR^{V}$, the Rayleigh quotient of $f$ is defined as
	\begin{align*}
		R_{(G,B)}(f) \coloneqq \frac{\langle f,\Delta f\rangle}{\langle f,f\rangle_B}=\frac{\sum\limits_{\{x,y\} \in E} (f(x)-f(y))^2 }{\sum\limits_{x \in B} f^2(x)},
	\end{align*}
	where the right-hand side is understood as $+\infty$ if $f|_{B}=0$.
	The variational characterization of the Steklov eigenvalues is given as
	\begin{align*}
		\sigma_k(G,B) &= \min_{W \subset \RR^V, \dim W = k} \max_{f \in W} R_{(G,B)},
	\end{align*}
	where $1\leq k\leq |B|$. 
    
%Combining with the variational characterization of the Steklov eigenvalues, we prove
The following lemma is useful to obtain the upper bounds of higher Steklov eigenvalues.
    
	\begin{lemma}\label{lem:Rf}
		Let $G$ be a graph with boundary $B$. For any $k\geq 2$, suppose that $v_1,v_2,\ldots,v_k\in \mathbb{R}^V$ is a collection of vectors such that $\mathrm{supp}(v_i)\cap B\neq \emptyset$ for any $i$ and  $\mathrm{supp}(v_i)\cap\mathrm{supp}(v_j)=\emptyset$ for all $1\leq i<j\leq k$. Then
		\begin{equation*}
			\sigma_k(G,B)\leq 2\max\limits_{1\leq i\leq k}\{R_{(G,B)}(v_i)\}.
		\end{equation*}
	\end{lemma}
	\begin{proof}
        By the definition, $W=\mathrm{span}\{v_1,\ldots,v_k\}$ is a $k$-dimensional subspace. For any $f=\sum_{i=1}^kc_iv_i\in W$ and $x,y\in V$, we claim that 
        \begin{align}\label{eq:Rf1}
            (f(x)-f(y))^2\leq 2\sum_{i=1}^kc_i^2(v_i(x)-v_i(y))^2.
        \end{align}
        If $f(x)=0$ or $f(y)=0$, there is nothing to prove. Then since $v_i$'s have disjoint supports, we can suppose that $p$ (resp. $q$) is the only index satisfying that $v_p(x)\neq 0$ (resp. $v_q(x)\neq 0$). In the case $p\neq q$,
        we have 
        \begin{align*}
            (f(x)-f(y))^2=&(c_pv_p(x)-c_qv_q(y))^2\\
            \leq& 2(c_p^2v_p(x)^2+c_q^2v_q(x)^2)\\
            \leq& 2\sum_{i=1}^kc_i^2(v_i(x)-v_i(y))^2.
        \end{align*}
        The case $p=q$ is direct. 
        Hence by \eqref{eq:Rf1}, we get
        \begin{align*}
            R_{(G,B)}(f)=&\frac{\sum\limits_{\{x,y\} \in E} (f(x)-f(y))^2 }{\sum\limits_{x \in B} f^2(x)}\\
            \leq& \frac{2\sum\limits_{i=i}^kc_i^2\sum\limits_{\{x,y\} \in E} (v_i(x)-v_i(y))^2 }{\sum\limits_{i=i}^kc_i^2\sum\limits_{x \in B} v_i^2(x)}\\
            \leq& 2\max\limits_{1\leq i\leq k}\{R_{(G,B)}(v_i)\}.
        \end{align*}
		Therefore, $\sigma_k(G,B)\leq 2\max\limits_{1\leq i\leq k}\{R_{(G,B)}(v_i)\}$.
	\end{proof}

    \subsection{Flows and congestion}
    Let $G = (V,E)$ be a finite graph.
	For every pair $u,v \in V$, let $\cP_{uv}$ be the set of paths between $u$ and $v$ in $G$. 
	Denote $\cP = \bigcup_{u,v\in V} \cP_{uv}$. 
	A  \textit{flow} in $G$ is a map $F : \mathcal{P} \longrightarrow \mathbb{R}_{\geq 0}$. For any $u,v\in V$, let $F[u,v]=\sum_{p\in\cP_{uv}}F(p)$ be the amount of flow sent between $u$ and $v$. 
	
	In this paper, we  consider a class of flows called \textit{probability flows}. Let $\mu$ be a probability distribution on subsets of $V$. $F$ is called a $\mu$\textit{-flow} if  $$F[u,v]=\mathbb{P}_{S\sim\mu}[u,v\in S]$$
    for all $u,v\in V$. Here the notation $S\sim\mu$ denotes that the subset $S$ is chosen according to the distribution $\mu$. For the boundary $B\subseteq V$ and $r\leq |B|$, denote the set of all $\mu$-flows with $\mathrm{supp}(\mu)\subseteq \binom{B}{r}$ by $\mathcal{F}_r(G,B)$.

	We define the \textit{congestion} of a flow $F$ by $\mathrm{con}(F) = {\sum_{v \in V} C_F(v)^2}$, where $C_F(v) = \sum_{p \in \cP : v \in p} F(p)$. This congestion can also be written as 
	\begin{align*}
		\mathrm{con}(F)=\sum_{p,p'\in\cP}\sum_{v\in p\cap p'}F(p)F(p').
	\end{align*}
The \textit{intersection number} is defined as
	\begin{align*}
		\mathrm{inter}(F)=\sum_{\substack{u,v,u',v'\\|\{u,v,u',v'\}|=4}}\sum_{\substack{p\in \cP_{uv}\\p'\in \cP_{u'v'}}}\sum_{x\in p\cap p'}F(p)F(p').
	\end{align*}
It is direct to check that $\con(F)\geq\inter(F)$.

An edge-weighted graph $H$ is a graph equipped with a nonnegative weight function $\omega: E(H)\rightarrow \mathbb{R}_{\geq 0}$ on edges.  If $\omega(e)\in\{0,1\}$ for every $e\in E(H)$, we call $H$  \textit{unit edge-weighted}.  We say that a flow $F$ in $G$  is an $H$\textit{-flow} if there exists an injective mapping $\phi:V(H)\rightarrow V(G)$ such that $F[\phi(u),\phi(v)]\geq\omega(u,v)$ for all $\{u,v\}\in E(H)$. In this case, $H$ and $G$ are respectively called the \textit{demand graph} and  the \textit{host graph}.  

Finally, with the definition of demand graph and host graph, we introduce $G$-intersection number.
\begin{definition}
		Let $G$ be an arbitrary host graph, and $H$ an edge-weighted demand graph. 
		The $G$-intersection number of $H$ is defined as
		\begin{align*}
			\mathrm{inter}_G(H)=\min_{\text{$F$ an $H$-flow in $G$}}\mathrm{inter}(F).
		\end{align*}
	\end{definition}

\subsection{Padded decomposition}
     Our proof is based on a graph semi-metric, which makes graph $G$ a semi-metric space defined as follows.
	\begin{definition}[Semi-metric space]
		Let $X$ be a nonempty set. 
		If the function $d:X \times X \rightarrow \mathbb{R}$ satisfies that 
		\begin{enumerate}
			% \begin{split}
				\item $d(x, y) \geq 0$;
				\item $d(x, y)=d(y, x)$;
				\item $d(x, y) \leq d(x, z)+d(z, y)$,
				% \end{split}&
		\end{enumerate}
		for any $x,y,z \in X$, 
		then we call $(X,d)$  a semi-metric space.
	\end{definition}
	
	\begin{definition}[Graph semi-metric]
		Let $G=(V,E)$ be a finite graph and $\omega : V\longrightarrow \mathbb{R}_{\geq 0}$ a nonnegative vertex weight. Then $\omega$ induces a finite semi-metric space via the function $d_{\omega} :V \times V \longrightarrow \mathbb{R}_{\geq 0}$, where $d_{\omega}(u, v) = min_{p \in \mathcal{P}_{uv}} \sum_{v \in p} \omega(v)$. We denote this  space by $(V,d_{\omega})$.
	\end{definition}
We also define the \textit{diameter} of $G$ in $(V,d_{\omega})$ as $diam(G)\coloneqq\max\limits_{u,v\in V(G)}d_\omega(u,v)$. 
	Next, we recall the  definitions of the padded decomposition and the modulus of padded decomposition.
	
	\begin{definition}[Padded decomposition] 
		Let $(X, d)$ be a finite metric space.  
		If $P$ is a partition of $X$, we will also regard it as a function $P: X \rightarrow 2^{X}$ such that for every $x \in X$, $P(x)$ is the unique part $C \in P$ for which $x \in C$.
		Let $\nu$ be a probability distribution over partitions of $X$, and $P$ a random partition distributed according to $\nu$. 
		We say that $P$ is $\kappa$-bounded if it holds that for every $C \in P$, we have $diam(C) \leq \kappa$.
		We say that $P$ is $(\alpha, \delta)$-padded if
		\begin{align*}
			\mathbb{P}[B(x, \kappa/\alpha) \subseteq P(x)] \geq \delta
		\end{align*}
		for every $x \in X$. 
		Here $B(x, r) \coloneqq \set{y \in X : d(x, y) \leq r}$. 
	\end{definition}
	
	\begin{definition}[The modulus of padded decomposibility]
		The modulus $\alpha(X, d; \delta, \kappa)$ of padded decomposibility of a finite metric space $(X, d)$ is defined as 
		\begin{align*}
			\alpha(X, d; \delta, \kappa) &= \inf \set{\alpha: %X \text{ admits a } %\kappa \text{-bounded } \alpha \text{-padded partition } \mu 
				(X, d) \text{ admits a } \kappa\text{-bounded } (\alpha, \delta)\text{-padded random partition}}. 
		\end{align*}
	\end{definition}

Here we give some notations for convenience.  Given two expressions $A$ and $B$, we write $A = \bigO_{p}(B)$ to express that $A \leq C_p B$ for some constant $C_p > 0$ which is independent of the variables in $B$ but may depend on $p$. 
	If the constant $C_p$ is universal, then we omit $p$ in the notation, namely we write $A = \bigO(B)$. 
	We also write $A \lesssim B$ as a synonym for $A=\bigO(B)$.

    Now we state two important theorems about the modulus of padded decomposibility in~\cite{biswal2010eigenvalue} and~\cite{lee2010genus}.
	
	\begin{theorem}[{\cite[Corollary 4.2]{biswal2010eigenvalue}}]\label{thm:minor-free-alpha-bound}
		Let $G=(V,E)$ be a $K_r$-minor-free graph. 
		Then for every $\kappa>0$ and nonnegative vertex weight $\omega$, $\alpha(V, d_{\omega}; \frac{1}{2}, \kappa)=\mathcal{O}(r^2)$. 
	\end{theorem}
	
	\begin{theorem}[{\cite[Theorem 4.1]{lee2010genus}}]\label{thm:genus-g-alpha-bound}
		Let $G=(V,E)$ be a graph of orientable genus $g$. Then for every $\kappa>0$ and nonnegative vertex weight $\omega$, $\alpha(V, d_{\omega}; \frac{1}{8}, \kappa)=\mathcal{O}(\mathrm{log}~g)$. 
	\end{theorem}

    As we all know, any planar graph is $K_5$-minor-free, which implies that $\alpha(V, d_{\omega}; \frac{1}{2}, \kappa)=\mathcal{O}(1)$ for every $\kappa>0$ and nonnegative vertex weight $\omega$.
	\section{Proofs}
	
	\subsection{Spreading weights}
	Graphs we consider in our proof are all finite.

	\begin{definition}
		Let $G=(V,E)$ be a  graph with boundary $B$ and  semi-metric $d_\omega$.
		Say that $\omega$ is $(r,\epsilon,B)$-spreading if we have 
		\begin{equation*}
			\frac{1}{|S|^2}\sum_{u,v\in S}d_{\omega}(u,v)\geq \epsilon\sqrt{\sum_{v\in V}\omega(v)^2}.
		\end{equation*}
		for every $S\subseteq B$ with $|S|=r$. Write $\epsilon_r(G,B,\omega)$ for the maximal value of $\epsilon$ for which $\omega$ is $(r,\epsilon,B)$-spreading.
	\end{definition}
	
	First, we give an estimate of the Steklov eigenvalue by the modulus of padded decomposability and spreading weight.
	\begin{theorem}\label{thm:key}
		Let $G$ be a graph with boundary $B$ and  maximum degree $D$.
        Let $\sigma_k$ be the $k$-th Steklov eigenvalue. 
        Then for any $0<\delta<1$, $2\leq k\leq\delta\card{B}/4$ and any vertex weight $\omega: V\rightarrow \mathbb{R}_{\geq 0}$ with 
		\begin{equation*}
			\sum_{v\in V}\omega(v)^2=1,
		\end{equation*}
		we have $$\sigma_k\leq \frac{128}{\delta}\frac{D}{|B|}(\frac{\alpha}{\epsilon})^2,$$
		where $\epsilon=\epsilon_{\floor{\delta|B|/4k}}(G,B,\omega)$ and $\alpha=\alpha(V,d_{\omega};\delta,\frac{\epsilon}{2})$.
	\end{theorem}
	\begin{proof}
		By the definition of $\alpha$, there exists a $\epsilon/2$-bounded random partition $P$ chosen according to $\nu$ such that 
		\begin{align*}
			\mathbb{P}[B(x, \epsilon/(2\alpha)) \subseteq P(x)] \geq \delta
		\end{align*}
		for every $x \in V$. 
		Define 
		\begin{align*}
			B_{i} \coloneqq \set{x \in B : B(x, \epsilon/(2\alpha)) \subseteq P_i(x)}
		\end{align*}
		and $a_i = \mathbb{P}(P = P_i)$. 
		We have 
		\begin{align*}
			\sum_{x\in B} \mathbb{P}[B(x, \epsilon/(2\alpha)) \subseteq P(x)] & =\sum_{x\in B}\sum_ia_i\mathbbm{1}_{\{B(x, \epsilon/(2\alpha))\subseteq P_i(x)\}}\\
			&=\sum_ia_i\sum_{x\in B}\mathbbm{1}_{\{B(x, \epsilon/(2\alpha))\subseteq P_i(x)\}}\\
			&=\sum_ia_i|B_i|\geq\delta\card{B}.
		\end{align*}
		Therefore, there exists a $\epsilon/2$-bounded partition $C$ such that the set
		\begin{align*}
			\hat{C}=\{x\in B: B(x, \epsilon/(2\alpha)) \subseteq C(x)\}
		\end{align*}
		satisfies $|\hat{C}|\geq\delta|B|$.
		Let $C=\{C_1, C_2,\cdots,C_m\}$, and for each $i\in [m]$, we define
		\begin{equation*}
			\hat{C}_i=\{x\in C_i\cap B: B(x,\epsilon/(2\alpha))\subseteq C_i\}.
		\end{equation*}
Note that some $\hat{C}_i$ may be empty.  We exclude those empty parts and relabel all the subscripts. 
		Then  we get
		\begin{equation*}
|\hat{C}_1\cup\hat{C}_2\cup\cdots\cup\hat{C}_{m'}|=|\hat{C}|\geq \delta |B|.
		\end{equation*}
		for some $m'\leq m$.
		
		For any set $A\subseteq B$ with $diam(A)\leq \epsilon/2$, we have 
		\begin{equation}\label{eq:1/(A^2)sum_dist}
			\frac{1}{|A|^2}\sum_{u,v\in A}d_{\omega}(u,v)\leq \frac{\epsilon}{2}=\frac{\epsilon}{2}\sqrt{\sum_{v\in V}\omega(v)^2}.
		\end{equation}
		        
        If $|\hat{C}_i|>\delta|B|/4k$, then we could pass to a subset of $\hat{C}_i$ of size exactly $\floor{\delta|B|/4k}$ which satisfies \eqref{eq:1/(A^2)sum_dist} since $diam(\hat{C}_i)\leq \epsilon/2$. However, this would lead to a contradiction with the definition of $\epsilon$, and then  we have $|\hat{C}_i|\leq \delta|B|/4k$ for each $i=1,2,\ldots,m'$. Therefore, by taking disjoint unions of the sets $\{\hat{C}_i\}$, we can find $2k$ sets $S_1,S_2,\ldots,S_{2k}$ with 
		\begin{equation*}
			\frac{\delta|B|}{2k}\geq |S_i|\geq \frac{\delta|B|}{4k}\geq 1.
		\end{equation*}
		For each $i\in [2k]$, let $\tilde{S}_i$ be the $\epsilon/(2\alpha)$-neighborhood of $S_i$. Observe  that the sets $\{\tilde{S}_i\}$ are pairwise disjoint, since by construction each is contained in a union of $C_i$'s, which are themselves pairwise disjoint.
		
		Now, for every $i\in [2k]$ we define
		\begin{equation*}
			W(\tilde{S}_i)\coloneqq\sum_{u\in \tilde{S}_i}\sum_{v:uv\in E}\left[\omega(u)+\omega(v)\right]^2.
		\end{equation*}
		Clearly, we have 
		\begin{equation*}
			\sum_{i=1}^{2k}W(\tilde{S}_i)\leq 2\sum_{uv\in E}\left[\omega(u)+\omega(v)\right]^2\leq 4D\sum_{v\in V}\omega(v)^2=4D.
		\end{equation*}
		 By renumbering the sets so that $\{\tilde{S}_1,\tilde{S}_2,\ldots,\tilde{S}_k\}$ have the smallest $W(\tilde{S}_i)$ values,   we have $W(\tilde{S}_i)\leq 4D/k$ for each $i=1,2,\ldots,k$.
		
		Finally, we define functions $f_1,f_2,\ldots,f_k:V\rightarrow\mathbb{R}$ by 
		\begin{equation*}
			f_i(x)=\mathrm{max}\left\{0,\frac{\epsilon}{2\alpha}-d_{\omega}(x,S_i)\right\}
		\end{equation*}
		so that $f_i$ is supported on $\tilde{S}_i$, and $f_i(x)=\epsilon/2\alpha$ for $x\in S_i\subseteq B$.
		Thus, these functions have disjoint support.
        Since each $f_i$ is $1$-Lipschitz and satisfies $\mathrm{supp}(f_i)\subseteq \tilde{S}_i$, we have 
		\begin{equation*}
			\begin{aligned}
				\sum_{uv\in E}|f_i(u)-f_i(v)|^2=&\sum_{u\in \tilde{S}_i}\sum_{v:\{u,v\}\in E}|f_i(u)-f_i(v)|^2\\
				\leq&\sum_{u\in \tilde{S}_i}\sum_{v:\{u,v\}\in E}d_{\omega}(u,v)^2\\
				=&\sum_{u\in \tilde{S}_i}\sum_{v:\{u,v\}\in E}\left[\omega(u)+\omega(v)\right]^2\\
				=&W(\tilde{S}_i)\leq \frac{4D}{k}.
			\end{aligned}
		\end{equation*}
		Furthermore, the functions satisfy
		\begin{equation*}
			\sum_{x\in B}f_i(x)^2\geq (\frac{\epsilon}{2\alpha})^2|S_i|\geq \frac{\epsilon^2}{16\alpha^2}\frac{\delta|B|}{k}.
		\end{equation*}

		Combining with two estimates above, we obtain
		\begin{equation*}
			R(f_i)\leq \frac{64}{\delta}\frac{D}{|B|}(\frac{\alpha}{\epsilon})^2
		\end{equation*}
		for each $f_i$, and the proof is complete due to~\cref{lem:Rf}.
	\end{proof}
	Next, we show a duality between the optimization problem of finding a spreading weight $\omega$ and the problem of minimizing congestion in probability flows. We call an optimization problem convex if it has the following form:
    \begin{align*}
			\text{minimize}~~~~~~~ &f_0(x)\\
			\text{subject to}~~~~~~&f_i(x)\leq b_i,\quad i=1,\ldots,m,
		\end{align*}
        where all the functions $f_i:\mathbb{R}^n\rightarrow\mathbb{R}$ ($i=0,\ldots,m$) are convex, which means $f_i(ax+by)\leq af_i(x)+bf_i(y)$ for all $x,y\in\mathbb{R}^n$ and all $a,b\in\mathbb{R}$ with $a+b=1$, $a\geq 0$, $b\geq 0$.
        
    Before our proof, we recall  Slater's condition in convex optimization, which can be found in Chapter 5 of \cite{boyd2004convex}.
    \begin{theorem}[Slater's condition]\label{Slater's condition}
        	When the feasible region for a convex problem $(\boldsymbol{P})$ has non-empty interior, the strong duality holds, which means the value of $(\boldsymbol{P})$ is equal to its dual $(\boldsymbol{P^*})$.
    \end{theorem}

	\begin{theorem}\label{thm:dual}
		Let $G=(V,E)$ be a graph with boundary $B$ and let $r\leq |B|$. Then
		\begin{equation*}
			\mathrm{max}\{\epsilon_r(G,B,\omega)~|~\omega:V\rightarrow\mathbb{R}_{\geq0}\}=\frac{1}{r^2}\mathrm{min}\{\sqrt{\mathrm{con}(F)}~|~F\in \mathcal{F}_r(G,B)\}.
		\end{equation*}
	\end{theorem}
	\begin{proof}
		Our goal is to  write out the optimizations $\max_{\omega}\epsilon_r(G,B,\omega)$ and $\frac{1}{r^2}\min_F\sqrt{\mathrm{con}(F)}$ as convex problems, and then show that they are dual to each other. 
		
		  First, we  expand $\max_{\omega}\epsilon_r(G,B,\omega)$ as a convex problem $(\boldsymbol{P})$. Let $P\in \{0,1\}^{\cP\times V}$ be the path incidence matrix, $Q\in\{0,1\}^{\cP\times \binom{V}{2}}$ the path connection matrix, and $R\in\{0,1/r^2\}^{\binom{B}{r}\times\binom{V}{2}}$ a normalized set containment matrix, which are respectively defined as 
		\begin{align*}
			P_{p,v}=\begin{cases}
				1, & v\in p, \\ 
				0, & \text{else},
			\end{cases}
			~~~~Q_{p,uv}=\begin{cases}
				1, & p\in \cP_{uv}, \\ 
				0, & \text{else},
			\end{cases}
			~~~~R_{S,uv}=\begin{cases}
				1/r^2, & \{u,v\}\subset S, \\ 
				0, & \text{else}.
			\end{cases}
		\end{align*}
		Then the convex problem $(\boldsymbol{P})=\max\limits_{\omega}\epsilon_r(G,B,\omega)$ can be written as 
		\begin{align*}
			\text{minimize}~~~~~~~ &-\epsilon\\
			~~~\text{subject to}~~~\epsilon\boldsymbol{1}&\preceq Rd~~~Qd\preceq Ps~~~s^{\top}s\leq 1\\
			d&\succeq0~~~~~~~~s\succeq0
		\end{align*}
		We introduce the nonnegative Lagrange multipliers $\mu,\lambda,\nu$, and get the Lagrangian function as follows: 
		\begin{align*}
			L(d,s,\mu,\lambda,\nu)=-\epsilon+\mu^{\top}(\epsilon\boldsymbol{1}-Rd)+\lambda^{\top}(Qd-Ps)+\nu(s^{\top}s-1).
		\end{align*}
		According to the theory of Lagrange dual shown in \cite{boyd2004convex}, $(\boldsymbol{P})$ and its dual $(\boldsymbol{P^*})$ have the following equivalent forms: 
		\begin{align*} &(\boldsymbol{P})=\inf_{\epsilon,d,s}\sup_{\mu,\lambda,\nu}L(d,s,\mu,\lambda,\nu),\\&(\boldsymbol{P}^*)=\sup_{\mu,\lambda,\nu}\inf_{\epsilon,d,s}L(d,s,\mu,\lambda,\nu).  
		\end{align*}
        By rearranging terms in $L$, we have 
		\begin{align*}
(\boldsymbol{P}^*)=&\sup_{\mu,\lambda,\nu}\inf_{\epsilon,d,s}[(\mu^{\top}\boldsymbol{1}-1)\epsilon+(\lambda^{\top}Q-\mu^{\top}R)d+(\nu s^{\top}s-\lambda^{\top}Ps)-\nu]\\
			=&\sup_{\mu,\lambda,\nu}[\inf_{\epsilon}(\mu^{\top}\boldsymbol{1}-1)\epsilon+\inf_d(\lambda^{\top}Q-\mu^{\top}R)d+\inf_s(\nu s^{\top}s-\lambda^{\top}Ps)-\nu].
		\end{align*}
		The infima $\inf_{\epsilon}(\mu^{\top}\boldsymbol{1}-1)\epsilon$ and $\inf_d(\lambda^{\top}Q-\mu^{\top}R)d$ are either $0$ or $-\infty$. Thus, at the optimum, they must be zero and $\mu^{\top}\boldsymbol{1}-1\geq 0$, $\lambda^{\top}Q-\mu^{\top}R\succeq0$. With these two constraints, the optimization reduces to 
        $$\sup_{\lambda,\nu}[\inf_s(\nu s^{\top}s-\lambda^{\top}Ps)-\nu].$$ At the optimum, the gradient of $\nu s^{\top}s-\lambda^{\top}Ps$ is zero, which leads to $s=P^{\top}\lambda/2\nu$. Now we obtain the infimum  $-\frac{\norm{P^{\top}\lambda}_2^2}{4\nu}$. Then, due to the mean value inequality, $\nu=\frac{1}{2}\norm{P^{\top}\lambda}_2$  at the maximum so that it is equivalent to considering the supremum of $-\norm{P^{\top}\lambda}_2$. Above all, 
 we have shown that $(\boldsymbol{P}^*)$ can be simplified as the 
 following convex problem
		\begin{align*}
			\text{maximize}~~~ -\norm{P^{\top}&\lambda}_2\\
			~~~\text{subject to}~~~\lambda^{\top}Q\succeq&\mu^{\top}R~~~\mu^{\top}\boldsymbol{1}\geq1\\
			\lambda\succeq&0~~~~~~~~\mu\succeq0
		\end{align*}
        To minimize $\norm{P^{\top}\lambda}_2$, this problem is equivalent to 
        \begin{align*}
			\text{maximize}~~~ -\norm{P^{\top}&\lambda}_2\\
			~~~\text{subject to}~~~\lambda^{\top}Q=&\mu^{\top}R~~~\mu^{\top}\boldsymbol{1}=1\\
			\lambda\succeq&0~~~~~~~~\mu\succeq0
		\end{align*}
		which is precisely the negative of the problem to minimize $\frac{1}{r^2}\sqrt{\mathrm{con}(F)}$, where $F$ is taken over all the maps in $\mathcal{F}_r(G,B)$. The proof is complete by Theorem \ref{Slater's condition}.
	\end{proof}
    
	\subsection{Proof of~\texorpdfstring{\cref{thm:sigma_k}}{Theorem }}
 Let's begin with the definition and some results of ($c,a$)-congestion measure.
	\begin{definition}
		Given a host graph $G$, we say that $\mathrm{inter}_G$ is a $(c,a)$-congestion measure if  we have the inequality 
		\begin{align*}
			\mathrm{inter}_G(H)\geq \frac{|E(H)|^3}{c|V(H)|^2}-a|V(H)|.
		\end{align*}
        for all unit-weighted graphs $H=(V,E)$.
	\end{definition}

	\begin{theorem}[{\cite[Theorem 4.1]{GAFA}}]\label{thm:interG}
		There is a universal constant $c_0>0$ such that the following holds. Let $\mu$ be any probability distribution on subsets of $[n]$. For $u,v\in [n]$, define 
		\begin{align*}
			F(u,v)\coloneqq\mathbb{P}_{S\sim\mu}[u,v\in S],
		\end{align*}
		and let $H_{\mu}$ be the graph on $[n]$ weighted by $F$. For any graph $G$ such that $\mathrm{inter}_G$ is a $(c,a)$-congestion measure, we have 
		\begin{align*}
			\mathrm{inter}_G(H_{\mu})\gtrsim \frac{1}{cn}(\mathbb{E}|S|^2)^{5/2}-c_0\frac{a}{n}\mathbb{E}|S|^2.
		\end{align*}
	\end{theorem}
	
	\begin{lemma}[{\cite[Corollary 3.6]{GAFA}}]\label{lem:congestion measure}
		If $G$ is a planar graph, then $\mathrm{inter}_G$ is an $(\mathcal{O}(1),3)$-congestion measure. If $G$ is of genus $g>0$, then $\mathrm{inter}_G$ is an $(\mathcal{O}(g),\mathcal{O}(\sqrt{g})$-congestion measure. If $G$ is $K_h$-minor-free, then $\mathrm{inter}_G$ is an $(\mathcal{O}(h^2\mathrm{log}~h),\mathcal{O}(h\sqrt{\mathrm{log}~h})$-congestion measure. 
	\end{lemma}
	
	Now we  give the proof of~\cref{thm:sigma_k}.
	\begin{proof}[Proof of~\cref{thm:sigma_k}]
		We prove the case with given genus, and   other cases follow similarly. Let $G=(V,E)$ be a graph with boundary $B$ of genus $g$ and maximum degree. First, for $2\leq k\leq\delta\card{B}/4$, by~\cref{thm:key}, we see that for any $0<\delta<1$ and any vertex weight  $\omega: V\rightarrow \mathbb{R}_{\geq 0}$ with 
		\begin{equation}\label{eq:sum_w(v)^2=1}
			\sum_{v\in V}\omega(v)^2=1,
		\end{equation}
		we have $$\sigma_k\leq \frac{128D}{\delta|B|}(\frac{\alpha}{\epsilon})^2,$$
		where $\epsilon=\epsilon_{\floor{\delta|B|/4k}}(G,B,\omega)$ and $\alpha=\alpha(V,d_{\omega};\delta,\frac{\epsilon}{2})$.
		\cref{thm:genus-g-alpha-bound} shows that  $$\alpha(V,d_{\omega};1/8,\frac{\epsilon}{2})=\mathcal{O}(\mathrm{log}~g)$$
        for any $\omega$ and $\epsilon$. Hence 
		\begin{align}\label{eq:key in mian proof}
			\sigma_k\lesssim \frac{D(\mathrm{log}~g)^2}{(\epsilon_{\floor{|B|/32k}}(G,B,\omega))^2|B|}
		\end{align}
		
		If $F\in \mathcal{F}_r(G,B)$, then there exists a probability distribution $\mu$ on subsets of $V$ such that $F$ is a $\mu$-flow and $\mathrm{supp}(\mu)\subseteq \binom{B}{r}$. Note that $\mu$ restricted to subsets of $B$ denoted by $\mu|_B$ is  a probability distribution over subsets of $B$ due to $\mathrm{supp}(\mu)\subseteq \binom{B}{r}$.
		Let $H_{\mu|_B}$ be the graph on $B$ weighted by $f(u,v)=\mathbb{P}_{S\sim\mu|_B}[u,v\in S]$. Then $F$ is a $H_{\mu|_B}$-flow in $G$ via the injective mapping $\phi:V(B)\rightarrow V(G)$ defined by $\phi(u)=u$. Therefore, we see that
		\begin{align*}
			\mathrm{inter}_G(H_{\mu|_B})=&\min_{\text{$F$ an $H_{\mu|_B}$-flow in $G$}}\mathrm{inter}(F)\\
			\leq&\min_{\text{$F$ an $H_{\mu|_B}$-flow in $G$}}\mathrm{con}(F)\\
			\leq& \min_{F\in \mathcal{F}_r(G,B)}\mathrm{con}(F).
		\end{align*}
		Using~\cref{thm:interG}, we have 
		\begin{align*}
			\min_{F\in \mathcal{F}_r(G,B)}\mathrm{con}(F)\geq& \mathrm{inter}_G(H_{\mu|_B})\\
			\gtrsim&\frac{1}{c|B|}(\mathbb{E}|S|^2)^{5/2}-c_0\frac{a}{|B|}\mathbb{E}|S|^2\\
			=&\frac{r^5}{c|B|}-c_0\frac{ar^2}{|B|},
		\end{align*}
		where the last step comes from $\mathrm{supp}(\mu)\subseteq \binom{B}{r}$.
		Hence by~\cref{lem:congestion measure}, there exists a constant $C_0> 0$ such that for any $C_0\sqrt{g}\leq r\leq |B|$, 
		\begin{align*}
			\min_{F\in \mathcal{F}_r(G,B)}\mathrm{con}(F)\gtrsim\frac{r^5}{g|B|}.
		\end{align*}
		Combining with \cref{thm:dual}, it implies that there exists a weight $\omega_r: V\rightarrow \mathbb{R}_{\geq 0}$ with $\epsilon_r(G,B,\omega_r)\gtrsim\frac{1}{r^2}\sqrt{r^5/g|B|}=\sqrt{r/g|B|}$ for $r\geq C_0\sqrt{g}$.
		
		 By using~\eqref{eq:key in mian proof}, we have 
		\begin{align*}
			\sigma_k\lesssim \frac{D(\mathrm{log}~g)^2}{(\epsilon_{r}(G,B,\omega_r))^2|B|}\lesssim\frac{D g(\mathrm{log}~g)^2}{r}\lesssim D g(\mathrm{log}~g)^2\frac{k}{|B|},
		\end{align*}
		for $r=\floor{|B|/32k}\geq C_0\sqrt{g}$.
        If $\floor{|B|/32k}< C_0\sqrt{g}$ or $k\geq\delta\card{B}/4$, then $\sigma_k=\mathcal{O}(D g(\mathrm{log}~g)^2\frac{k}{|B|})$ holds trivially due to the bound $\sigma_k\leq 2D$ obtained by  Lemma \ref{lem:Rf} for all $1\leq k\leq |B|$. 
    
    Our proof is complete.
	\end{proof}
    \section*{Declaration of competing interest}

The authors declare that they have no known competing financial interests or personal relationships that could have appeared to influence the work reported in this paper.

\section*{Data availability}

No data was used for the research described in the article.
	%\section*{Acknowledgements}

	\bibliographystyle{abbrv}
	\bibliography{main}
	
\end{document}